\documentclass[twoside]{article}

\usepackage{amsmath,amsthm,amssymb}

\usepackage{newtxtext,newtxmath}

\usepackage[text={12.5cm,19cm},centering,paperwidth=17cm,paperheight=24cm]{geometry}

\usepackage[fontsize=10.4pt]{scrextend}

\def\titlerunning#1{\gdef\titrun{#1}}

\makeatletter
\def\author#1{\gdef\autrun{\def\and{\unskip, }#1}\gdef\@author{#1}}
\def\address#1{{\def\and{\\\hspace*{15.6pt}}\renewcommand{\thefootnote}{}\footnote{#1}}\markboth{\autrun}{\titrun}}
\makeatother

\def\email#1{email: \href{mailto:#1}{#1} }
\def\subjclass#1{\par\bigskip\noindent\textbf{Mathematics Subject Classification 2020.} #1}
\def\keywords#1{\par\smallskip\noindent\textbf{Keywords.} #1}

\newenvironment{acknowledgments}{\bigskip\small\noindent\textit{Acknowledgments.}}{\par}

\newtheorem{thm}{Theorem}[section]
\newtheorem{cor}[thm]{Corollary}
\newtheorem{lem}[thm]{Lemma}

\newtheorem{conjecture}[thm]{Conjecture}

\newtheorem{mainthm}[thm]{Main Theorem}

\theoremstyle{definition}
\newtheorem{defin}[thm]{Definition}
\newtheorem{exa}[thm]{Example}

\newtheorem*{rem}{Remark}

\numberwithin{equation}{section}

\usepackage[hyperfootnotes=false,colorlinks=true,allcolors=blue]{hyperref}

\newcommand\Acal{\mathcal{A}}

\newcommand\Ccal{\mathcal{C}}
\newcommand\Dcal{\mathcal{D}}

\newcommand\Gcal{\mathcal{G}}

\newcommand\Ocal{\mathcal{O}}

\newcommand\Ucal{\mathcal{U}}

\newcommand\bx{\mathbf{x}}
\newcommand\by{\mathbf{y}}

\newcommand\bz{\mathbf{z}}

\newcommand\bH{\mathbf{H}}
\newcommand\bN{\mathbf{N}}

\newcommand\A{\mathbb{A}}
\newcommand\B{\mathbb{B}}
\newcommand\C{\mathbb{C}}
\newcommand\D{\overline{\mathbb D}}
\newcommand\CP{\mathbb{CP}}

\renewcommand\D{\mathbb D}

\newcommand\N{\mathbb{N}}

\newcommand\R{\mathbb{R}}

\newcommand\Z{\mathbb{Z}}

\newcommand\ggot{\mathfrak{g}}

\newcommand\igot{\mathfrak{i}}

\renewcommand\igot{\mathfrak{i}}

\newcommand\pgot{\mathfrak{p}}

\newcommand\E{\mathrm{e}}

\renewcommand\imath{\igot}

\newcommand\hra{\hookrightarrow}
\newcommand\lra{\longrightarrow}

\newcommand\wt{\widetilde}
\newcommand\wh{\widehat}
\newcommand\di{\partial}

\newcommand\TC{\mathrm{TC}}
\newcommand\CMI{\mathrm{CMI}}
\newcommand\CMInf{\mathrm{CMI}_{\mathrm{nf}}}
\newcommand\NC{\mathrm{NC}}
\newcommand\NCnf{\mathrm{NC}_{\mathrm{nf}}}

\newcommand\Area{\mathrm{Area}}

\usepackage[utf8]{inputenc}
\usepackage{hyperref}
\usepackage{verbatim}

\usepackage{enumerate}

\input xy
\xyoption{all}

\begin{document}

\titlerunning{Minimal surfaces in Euclidean spaces by way of complex analysis}

\title{\textbf{Minimal surfaces in Euclidean spaces \\ by way of complex analysis}}

\author{Franc Forstneri\v c}

\date{}

\maketitle

\address{University of Ljubljana, Faculty of Mathematics and Physics, 
Jadranska 19, SI--1000 Ljubljana, Slovenia, and
Institute of Mathematics, Physics and Mechanics, Jadranska 19, SI--1000 Ljubljana, Slovenia; \email{franc.forstneric@fmf.uni-lj.si}}



\begin{abstract}
This is an expanded version of my plenary lecture at 
the 8th European Congress of Mathematics in Portoro\v z  
on 23 June 2021. The main part of the paper is a survey of recent applications of 
complex-analytic techniques to the theory of conformal minimal surfaces in Euclidean spaces.
New results concern approximation, interpolation, and general position properties
of minimal surfaces, existence of minimal surfaces with a given Gauss map, 
and the Calabi--Yau problem for minimal surfaces. To be accessible to a wide audience,
the article includes a self-contained elementary introduction to the theory
of minimal surfaces in Euclidean spaces.

\end{abstract}

\tableofcontents

%
%
\section{Minimal surfaces: a link between mathematics, science, engineering, and art}\label{sec:intro}

Minimal surfaces are among the most beautiful and aesthetically pleasing geometric objects.
These are surfaces in space which locally minimize area, in the sense that any small enough piece
of the surface has the smallest area among surfaces with the same boundary.
From the physical viewpoint, these are surfaces minimizing tension, hence in equilibrium position.
They appear in a variety of applications to engineering, biology, architecture, and others.

The subject has a luminous history, going back to 1744 when Leonhard Euler \cite{Euler1744}
showed that pieces of the surface now called {\em catenoid} (see Example \ref{ex:catenoid})
have smallest area among all surfaces of rotation in the 3-dimensional Euclidean 
space $\R^3$. The catenoid derives it name from {\em catenary}, the curve that an idealized 
hanging chain assumes under its own weight when supported only at its ends.
The model catenary is the graph of the hyperbolic cosine function $y=\cosh x$, and a 
catenoid is obtained by rotating this curve around the $x$-axis in the $(x,y,z)$-space.
Topologically, a catenoid is a cylinder, and as a conformal surface it is the puncture plane 
$\C^*=\C\setminus\{0\}$. 
From mathematical viewpoint, the catenoid is one of the most paradigmatic examples of
minimal surfaces, and it appears in several important classification results and in proofs of major theorems.

The subject of minimal surfaces was put on solid footing by Joseph--Louis Lagrange who 
developed the calculus of variations during 1760-61, thereby reducing the problem 
of finding stationary points of functionals to a second order partial differential equation,
now called Lagrange's equation.
His work was published in 1762 by Accademia delle scienze di Torino 
\cite{Lagrange1762Paris1,Lagrange1762Paris2} and is available in his collected works 
\cite{LagrangeOeuvres1}. In the second paper \cite{Lagrange1762Paris2}, 
Lagrange applied his new method to a variety of problems in physics, dynamics, and geometry.
In particular, he derived the {\em equation of minimal graphs}. 
The term {\em minimal surface} has since been used for a surface which is a stationary point 
of the area functional. The question whether a domain in a minimal surface 
truly minimizes the area among nearby surfaces with the same boundary 
can be analysed by considering the second variation of area. 
It was later shown that a minimal graph in $\R^3$ over a compact convex domain in $\R^2$
is an absolute area minimizer, and hence small enough pieces of any minimal surface
are area minimizers.

In 1776, Jean Baptiste Meusnier \cite{Meusnier1776} discovered 
that domains in a surface in $\R^3$ are minimal in the sense of Lagrange if and only if the surface has  
vanishing mean curvature at every point.  He also described the second known minimal surface, 
the {\em helicoid}; see Example \ref{ex:helicoid}. 
It is obtained by a line in $3$-space rotating at a constant rate as it moves at a 
constant speed along the axis of rotation, which is perpendicular to the rotating line. 
Helicoid is the geometric shape of a device known as {\em Archimedes' screw}  
(or the water screw, screw pump, or Egyptian screw), named after Greek philosopher 
and mathematician Archimedes who described it around 234 BC on the occasion of his visit to Egypt. 
There is evidence that this device had been used in ancient Egypt much earlier.
The helicoid is sometimes called "double spiral staircase" --- each of the two 
half-lines sweeps out a spiral staircase, and these two staircases only meet along the axis
of rotation. Therefore, its physical model is a convenient device for letting people ascend and descend 
a staircase without the two crowds meeting in-between. 
From a different field, DNA molecules assume the shape of a helicoid. 

Topologically and conformally the helicoid is the plane. 
Its name derives from helix --- for every point on the helicoid, 
there is a helix (a spiral curve) contained in the helicoid which passes through that point.
The helicoid plays a major role in the classification
of properly embedded minimal surfaces in $\R^3$;  see the survey paper 
\cite{ColdingMinicozzi2006PNAS} by Tobias H.\ Colding and William P.\ Minicozzi. 

Minimal surfaces appear naturally in the physical world. Laws of physics imply that a soap
film spanned by a given frame (i.e., a closed Jordan curve) is a minimal surface.
The reason is that this shape minimizes the surface tension and puts it in equilibrium position.
Soap films, bubbles, and surface tension were studied by the Belgian physicist 
Joseph Plateau in the 19th century.
Based on his experiments, Karl Weierstrass formulated in 1873 the {\em Plateau problem},
conjecturing that any closed Jordan curve in $\R^3$ spans a minimal surface (in fact, a minimal disc). 
This was confirmed by Tibor Rad{\'o} \cite{Rado1930,Rado1930MZ} (1930) 
and Jesse Douglas \cite{Douglas1931} (1931). 
For his work on the Plateau problem, Douglas received one of the first two Fields Medals
at the International Congress of Mathematicians in Oslo in 1936. 
Half a century later, it was shown that the disc of smallest area with given boundary curve  
(the Douglas--Morrey solution of the Plateau problem) has no branch points; 
see the monograph by Anthony Tromba \cite{Tromba2012}.
Furthermore, if the curve lies in the boundary of a convex domain in $\R^3$ then the solution is embedded
according to  William H.\ Meeks and Shing Tung Yau \cite{MeeksYau1982T,MeeksYau1982MZ}. 

Minimal surfaces are also studied in more general Riemannian manifolds of dimension at least three.
Holomorphic curves in complex Euclidean spaces $\C^n$ for $n>1$, or in any 
complex K\"ahler manifold of complex dimension at least two, 
are special but important examples of minimal surfaces.
As pointed out by Colding and Minicozzi \cite{ColdingMinicozzi2006PNAS}, there are several
fields where minimal surfaces are actively used in understanding physical phenomena. 
In particular, they come up in the study of compound polymers, protein folding, etc. 
They also play a prominent role in art, especially in architecture.

The connection between minimal surfaces in Euclidean spaces and complex analysis 
has been known since mid-19th century. The basic fact is that a conformal immersion
$X:M\to \R^n$ from a Riemann surface $M$ parameterizes a minimal surface if and 
only if the map $X$ is harmonic (see Theorem \ref{th:harmonic}); 
equivalently, the complex derivative $\di X/\di z$  in any local
holomorphic coordinate $z$ on $M$ is holomorphic. 
Furthermore, the immersion $X$ is conformal if and only if $\di X/\di z$ 
assumes values in the null quadric  $\A\subset \C^n$, 
given by the equation $z_1^2+z_2^2+\cdots+z_n^2=0$  (see \eqref{eq:nullquadric}), 
and $\di X/\di z\ne 0$ if $X$ is an immersion. This leads to the 
{\em Enneper--Weierstrass representation} of any 
conformally immersed minimal surface $M\to\R^n$ as the real part of the integral of a holomorphic map 
$f:M\to \A_* = \A\setminus \{0\} \subset \C^n$ (see Theorem \ref{th:EW}). 
The period vanishing conditions on $f$ along closed curves in $M$ ensure that the 
integral is well-defined. The formula is most concrete in dimension $n=3$ (see \eqref{eq:EWR3}) 
due to an explicit 2-sheeted parameterization of the null quadric  $\A\subset\C^3$ by $\C^2$. 

This connection between minimal surfaces and holomorphic maps
was used by Bernhard Riemann around 1860 in his
construction of properly embedded minimal surfaces in $\R^3$, now called 
{\em Riemann's minimal examples} \cite{Riemann1868} (see the paper \cite{MeeksPerez2016}
by William H.\ Meeks and Joaqu\'in P\'erez), and in numerous further works by other authors. 
It was popularized again in modern times by Robert Osserman \cite{Osserman1986}. 

Despite the long and illustrious history of the subject, the author in collaboration with 
Antonio Alarc\'on, Francisco J.\ L\'opez and others obtained  in the last decade a string of new results
by exploiting the Enneper--Weierstrass representation. 
The main point in our approach is that the punctured null quadric $\A_*$ 
is a complex homogeneous manifold, hence an {\em Oka manifold}, a notion
introduced in \cite{Forstneric2009CR} and treated in \cite[Chapter 5]{Forstneric2017E}.
This implies that holomorphic maps from any open Riemann surface  
(and, more generally, from any Stein manifold, 
that is, a closed complex submanifold of a complex Euclidean space $\C^N$) 
to  $\A_*$ satisfy the Runge--Mergelyan approximation
theorem and the Weierstrass interpolation theorem in the absence of topological obstructions.
Together with methods of convexity theory, 
this gave rise to many new constructions of conformal minimal surfaces with interesting properties;
see Theorem \ref{th:approximation}.
By using parametric versions of these results, it was possible to determine 
the rough topological shape (i.e., the weak or strong homotopy type) of the space
of nonflat conformal minimal immersions from any given open Riemann surface into $\R^n$
(see Theorem \ref{th:structure}). 
It was also shown that every natural candidate is the Gauss map
of a conformal minimal surface in $\R^n$ (see Theorem \ref{th:Gauss}).

Another complex analytic technique, which has recently had a major impact 
on the field, is an adaptation of the classical Riemann--Hilbert boundary value problem to 
conformal minimal surfaces and holomorphic null curves in Euclidean spaces. This
led to an essentially optimal solution of the {\em Calabi--Yau problem for minimal surfaces}, 
originating in conjectures of Eugenio Calabi from 1965; see Theorems
\ref{th:CY0} and \ref{th:CY}. This technique was also  used 
in the construction of complete proper minimal surfaces in minimally convex 
domains of $\R^n$ (see \cite[Chapter 8]{AlarconForstnericLopez2021}).

The recent results  presented in Section \ref{sec:survey} 
are carefully explained in the monograph \cite{AlarconForstnericLopez2021} published in 
March 2021. The corresponding developments on non-orientable minimal surfaces are 
described in the AMS Memoir \cite{AlarconForstnericLopezMAMS} from 2020.
It is needless to say that both these publication contain many other results not mentioned here.

In 2021, David Kalaj and the author \cite{ForstnericKalaj2021} obtained an optimal Schwarz--Pick lemma for
conformal minimal discs in the ball of $\R^n$ and introduced the notion of hyperbolicity of domains in $\R^n$, 
in analogy with Kobayashi hyperbolicity of complex manifolds. 
This new topic is currently being developed, and it is too early to include it here.

%
%
%
%
\section{An elementary introduction to minimal surfaces}\label{sec:elementary}
To make the article accessible to a wide audience including advanced 
undergraduate students of Mathematics, we present in this section a self-contained introduction 
to the theory of minimal surfaces in Euclidean spaces.
We assume familiarity with elementary calculus, topology, and rudiments
of complex analysis; however, no a priori knowledge of differential geometry is expected. 
We shall use the fact that metric-related quantities such as length, area, and curvature of 
curves and surfaces in a Euclidean space $\R^n$ are invariant under translations and 
orthogonal maps of $\R^n$; 
these are the isometries of the Euclidean metric, also called {\em rigid motions}.
For simplicity  of presentation, we focus on minimal surfaces parameterized by plane domains, 
although the same methods apply on an arbitrary open Riemann surface.
More complete treatment is available in a number of texts; see 
\cite{Lawson1980,Osserman1986,BarbosaColares1986,Nitsche1989,ColdingMinicozzi1999,ColdingMinicozzi2011,MeeksPerez2011,MeeksPerez2012Survey,AlarconForstnericLopez2021}, among others. 
For the theory of non-orientable minimal surfaces, see \cite{AlarconForstnericLopezMAMS}.

%
%
\subsection{Conformal maps and conformal structures on surfaces.} \label{ss:conformal}
From the physical viewpoint, the most natural parameterization of a minimal surface 
is by a {\em conformal map} (from a plane domain, or a conformal surface). 
A conformal parameterization minimizes the total energy of the map and makes 
the tension uniformly spread over the surfaces. We give a brief
introduction to the subject of conformal maps, referring to 
\cite[Sections 1.8--1.9]{AlarconForstnericLopez2021} for more details and further references.

Let $D$ be a domain in $\R^2$ with coordinates $(u,v)$. A $\Ccal^1$ map
$X:D\to\R^n$ $(n\ge 2)$ is an {\em immersion} if the partial derivatives
$X_u=\di X/\di u$ and $X_v=\di X/\di v$ are linearly independent at every point of $D$.
An immersion is said to be {\em conformal} if its differential $dX_p$ at any point 
$p\in D$ preserves angles. It is elementary to see 
(cf.\ \cite[Lemma 1.8.4]{AlarconForstnericLopez2021}) that an immersion $X$ is conformal if and only if
\begin{equation}\label{eq:conformal}
	|X_u|=|X_v| \ \ \text{and}\ \ X_u\,\cdotp X_v=0.
\end{equation}
Here, $\bx\,\cdotp\by$ denotes the Euclidean inner product between vectors
$\bx,\by\in\R^n$ and $|\bx|=\sqrt{\bx\,\cdotp\bx}$ is the Euclidean length of $\bx$. 
A smooth map $X:D\to\R^n$ (of class $\Ccal^1$, not necessarily an immersion) is called conformal if 
\eqref{eq:conformal} holds at each point. It clearly follows that $X$ has rank zero at 
non-immersion points. 

Let $M$ be a topological surface. A {\em conformal structure} on $M$ is given by an atlas 
$\Ucal=\{(U_i,\phi_i)\}_{i\in I}$ with charts $\phi_i:U_i \stackrel{\cong}{\to} V_i\subset\R^2$ 
whose transition maps 
\[
	\phi_{i,j}=\phi_i\circ \phi_j^{-1}:\phi_j(U_i\cap U_j) \to \phi_i(U_i\cap U_j)
\] 
are conformal diffeomorphisms of plane domains. Identifying $\R^2$ with the 
complex plane $\C$, each map $\phi_{i,j}$ is biholomorphic or anti-biholomorphic.
A surface $M$ endowed with a conformal structure
(more precisely, with an equivalence class of conformal structures) 
is a {\em conformal surface}.  
If $M$ is orientable, then by choosing the charts $\phi_i$ in a conformal atlas
to preserve orientation, the transition maps $\phi_{i,j}$ are biholomorphic; hence, 
$\Ucal$ is a complex atlas and $(M,\Ucal)$ is a {\em Riemann surface}.
A connected non-orientable conformal surface $M$ admits a two-sheeted conformal 
covering $\wt M\to M$ by a Riemann surface $\wt M$. 

Assume now that $g$ is a {\em Riemannian metric} on a smooth surface $M$, i.e., 
a smoothly varying family of scalar products $g_p$ on tangent spaces $T_p M$, $p\in M$. 
In any local coordinate $(u,v)$ on $M$, the metric $g$ has an expression 
\[
	g =  E du^2 + 2F du dv + G dv^2,  
\]
where the coefficient functions $E,F,G$ satisfy $EG-F^2>0$. A local chart $(u,v)$ is said to be
{\em isothermal} for $g$ if the above expression simplifies to 
\[
	g = \lambda(u,v)\, (du^2+dv^2) = \lambda |dz|^2,\quad\ z=u+\imath v
\]
for some positive function $\lambda$. 
An important result, first observed by Carl Friedrich Gauss, is that in a neighbourhood 
of any point of $M$ there exist smooth isothermal coordinates. One way to obtain such
coordinates is from solutions of the classical  {\em Beltrami equation}. We refer to 
\cite[Secs.\ 1.8--1.9]{AlarconForstnericLopez2021} for a more precise statement and references.
Since the transition map between any pair of isothermal charts is a conformal diffeomorphism, 
we thus obtain a conformal atlas on $M$ consisting of isothermal charts. 
The upshot is that every Riemannian metric on a smooth surface determines a conformal structure. 
Furthermore, a pair of Riemannian metrics $g,\tilde g$ on $M$
determine the same conformal structure if and only if $\tilde g=\mu g$ for a smooth 
positive function $\mu$ on $M$.

Denote by $\bx=(x_1,\ldots,x_n)$ the Euclidean coordinates on $\R^n$ and by  
\[
	ds^2=dx_1^2+\cdots+dx_n^2
\] 
the Euclidean metric. If $X=(X_1,\ldots,X_n):M\to\R^n$ is a smooth immersion, then
\[
	g=X^*(ds^2) = (dX_1)^2+\cdots + (dX_n)^2
\]
is a Riemannian metric on $M$, called the {\em first fundamental form}. 
By the definition of $g$, the map $X:(M,g)\to (\R^n,ds^2)$ is an isometric immersion. 
By what has been said, $g$ determines a conformal structure on $M$ 
(assuming now that $M$ is a surface), and in this structure the map 
$X$ is a conformal immersion. More precisely, $X(u,v)$ is conformal in any
isothermal local coordinate $(u,v)$ on $M$.

This shows that any immersion $X:M\to\R^n$ from a smooth surface determines 
a unique conformal structure on $M$ which makes $X$ a conformal immersion.
If in addition $M$ is oriented, we get the structure of a Riemann surface. Results of conformality theory 
imply that if $D$ is a domain in $\R^2$ and $X:D\to\R^n$ is an immersion, then there is 
a diffeomorphism $\phi:D'\to D$ from another domain $D'\subset\R^2$ 
such that the immersion $X\circ\phi:D'\to \R^n$ is conformal. In particular,
if $D$ is the disc then we may take $D'=D$. 

The same arguments and conclusions apply to immersions of a smooth surface $M$ into 
an arbitrary Riemannian manifold $(N,\tilde g)$ in place of $(\R^n,ds^2)$.

%
%
\subsection{First variation of area and energy}\label{ss:variation}

Assume that $D\subset \R^2_{(u,v)}$ is a bounded domain with piecewise smooth boundary and
$X:\overline D\to \R^n$ is a smooth immersion. Precomposing $X$ with a diffeomorphism 
from another such domain in $\R^2$, we may assume that $X$ is conformal; see  
\eqref{eq:conformal}. We consider the {\em area functional}
\begin{equation}\label{eq:area}
	 \Area(X) = \int_D |X_u\times X_v| \, dudv
	= \int_D \sqrt{|X_u|^2 |X_v|^2 - |X_u\,\cdotp  X_v|^2} \, dudv
\end{equation}
and the {\em Dirichlet energy functional} 
\begin{equation}\label{eq:Dirichlet}
	 \Dcal(X) = \frac12 \int_D |\nabla X|^2 \, dudv 
	= \frac12  \int_D  \left(|X_u|^2+ |X_v|^2\right) \, dudv.
\end{equation}
We have elementary inequalities
\[
	 |\bx|^2|\by|^2 - |\bx\,\cdotp \by|^2 \le |\bx|^2|\by|^2  \le 
	 \frac14 \left(|\bx|^2+|\by|^2\right)^2,\quad \bx,\by\in\R^n,
\]
which are equalities if and only if $\bx,\by$ is a conformal frame, i.e., 
$|\bx|=|\by|$ and $\bx\,\cdotp \by=0$. Applying this to the vectors 
$\bx=X_u$ and $\by=X_v$ gives $\Area(X) \le \Dcal(X)$, with equality if and only if $X$ is conformal.
Hence, these two functionals have the same critical points on the set of conformal immersions.

It is elementary to find critical points of these functional. The calculation is 
simpler for the Dirichlet functional $\Dcal$, but the expression for the first variation 
is the same for both functionals at a conformal map $X$. 
Assuming that $G:\overline D\to \R^n$ is a smooth map vanishing on $bD$, 
the first variation of $\Dcal$ at $X$ in direction $G$ equals
\begin{equation}\label{eq:firstvariationD}
	\frac{d}{dt}\Big|_{t=0} \Dcal(X+tG) = 
	\int_D \left(X_u\,\cdotp G_u + X_v\,\cdotp G_v\right) dudv 
	= - \int_D \Delta X\, \cdotp G \, dudv,
\end{equation}
where $\Delta X = X_{uu} + X_{vv}$
is the Laplace of $X$. (We integrated by parts and used $G|_{bD}=0$.) 
The right-hand-side of \eqref{eq:firstvariationD} 
vanishes for all $G$ if and only if $\Delta X=0$. This proves:

%
%
\begin{thm}\label{th:harmonic}
Let $D$ be a relatively compact domain in $\R^2$ with piecewise smooth boundary.
A smooth conformal immersion $X:\overline D\to\R^n$ $(n\ge 3)$ 
is a stationary point of the area functional \eqref{eq:area} if and only if $X$ is harmonic: $\Delta X=0$.
\end{thm}

For completeness, we also calculate the first variation of area at a conformal immersion 
$X$. Let $G:\overline D\to\R^n$ be as above.  Consider the expression
under the integral \eqref{eq:area} for the map $X_t=X+tG$, $t\in \R$. 
Taking into account \eqref{eq:conformal} we obtain
\[
	|X_u+tG_u|^2 \,\cdotp |X_v+tG_v|^2 = 
	|X_u|^4 + 2t \left(X_u\,\cdotp G_u + X_v\,\cdotp G_v\right) |X_u|^2 + O(t^2),
\]
\[
	|(X_u+t G_u)\,\cdotp  (X_v+t G_v)|^2 = O(t^2).
\]
It follows that 
\begin{multline*}
	\frac{d}{dt}\Big|_{t=0} \bigl(|X_u+tG_u|^2 |X_v+tG_v|^2 - |(X_u+tG_u)\,\cdotp (X_v+tG_v)|^2 \bigr)  \cr
	= 2  |X_u|^2 \left(X_u\,\cdotp G_u + X_v\,\cdotp G_v\right) 
\end{multline*}
and therefore
\[
	\frac{d}{dt}\Big|_{t=0} \Area(X+tG) 
	= \int_D \left( X_u\,\cdotp G_u + X_v\,\cdotp G_v\right)  dudv 
	= -\int_D \Delta X \cdotp G \, dudv.
\]
(We integrated by parts and used that $G|_{bD}=0$.
The factor $2|X_u|^2$ also appears in the denominator when differentiating the 
expression for $\Area(X+tG)$ at $t=0$, so this term cancels.)
Comparing with \eqref{eq:firstvariationD}, we see that
\[
	\frac{d}{dt}\Big|_{t=0} \Area(X+tG)  = \frac{d}{dt}\Big|_{t=0} \Dcal(X+tG) = 
	-\int_D \Delta X \cdotp G \, dudv
\]
if $X$ is a conformal immersion.

The same result holds on any compact domain with piecewise smooth boundary in 
a conformal surface $M$. A conformal diffeomorphism changes the Laplacian by 
a multiplicative factor, so there is a well-defined notion of a harmonic function on $M$.

%
%
\subsection{Characterization of minimality by vanishing mean curvature}\label{ss:meancurvature}
In this section, we prove a result due to Meusnier \cite{Meusnier1776} which characterizes
minimal surfaces in terms of vanishing mean curvature; see Theorem \ref{th:CMI}.

To explain the notion of curvature of a smooth plane curve $C\subset \R^2$ at a 
point $p\in C$, we apply a rigid change of coordinates in $\R^2$ taking $p$ to $(0,0)$ and the
tangent line $T_pC$ to the $x$-axis, so locally near $(0,0)$ the curve is
the graph $y=f(x)$ of a smooth function on an interval around $0\in\R$,  
with $f(0)=f'(0)=0$. Therefore,  
\begin{equation}\label{eq:graphC}
	y= f(x) = \frac{1}{2} f''(0) x^2 + o(x^2).
\end{equation}
Let us find the circle which agrees with this graph to the second order at $(0,0)$.
Clearly, such a circle has centre on the $y$-axis, so it is of the form 
$x^2+(y-r)^2=r^2$ for some $r\in\R\setminus \{0\}$, unless $f''(0)=0$ when the $x$-axis  
(a circle of infinite radius) does the job. Solving the equation on $y$ near $(0,0)$ gives
\[
	y=r-\sqrt{r^2-x^2}= r-r\sqrt{1-\frac{x^2}{r^2}} 
	= r-r \left(1 -  \frac{x^2}{2r^2} + o(x^2) \right)  
	= \frac{1}{2r} x^2+ o(x^2).
\] 
A comparison with \eqref{eq:graphC} shows that for $f''(0)\ne 0$ the number
$
	r=1/f''(0) \in \R\setminus \{0\}
$
is the unique number for which the circle agrees with the curve \eqref{eq:graphC}
to the second order at $(0,0)$. This best fitting circle is called the {\em osculating circle}.
The number 
\begin{equation}\label{eq:kappa}
	\kappa=f''(0)=1/r 
\end{equation}
is the {\em signed curvature} of the curve \eqref{eq:graphC} at $(0,0)$, 
its absolute value $|\kappa|=|f''(0)|\ge 0$ is the {\em curvature},
and $|r|=1/|\kappa|=1/|f''(0)|$ is the {\em curvature radius}. If $f''(0)=0$ then
the curvature is zero and the curvature radius is $+\infty$.

Consider now a smooth surface $S\subset \R^3$.
Let $(x,y,z)$ be coordinates on $\R^3$.  Fix a point $p\in S$. A rigid change of coordinates gives
$p=(0,0,0)$ and $T_pS=\{z=0\}=\R^2\times \{0\}$. Then, $S$ is locally near the origin 
a graph of the form 
\begin{equation}\label{eq:graphS}
	z=f(x,y)= \frac{1}{2}\left( f_{xx}(0)x^2 + 2f_{xy}(0,0)xy + f_{yy}(0)y^2\right) + o(x^2+y^2).
\end{equation}
The symmetric matrix 
\begin{equation}\label{eq:Hessian}
	A=\left(\begin{matrix} f_{xx}(0,0) &  f_{xy}(0,0) \\ f_{xy}(0,0)  & f_{yy}(0,0) \end{matrix}\right)
\end{equation}
is called the {\em Hessian matrix} of $f$ at $(0,0)$.
Given a unit vector $v=(v_1,v_2)$ in the $(x,y)$-plane, 
let $\Sigma_v$ be the $2$-plane through $0\in \R^3$ spanned by $v$ and the $z$-axis.
The intersection $C_v := S\cap \Sigma_v$ is then a planar curve contained in $S$, given by
\begin{equation}\label{eq:Cv}
	z= f(v_1t,v_2t) = \frac{1}{2} (Av\,\cdotp v) \, t^2  + o(t^2) 
\end{equation}
for $t\in \R$ near $0$. Since $|v|=1$, the parameters $(t,z)$ on $\Sigma_v$ 
are Euclidean parameters, i.e., the Euclidean metric $ds^2$ on $\R^3$ restricted to the plane $\Sigma_v$
is given by $dt^2+dz^2$. From our discussion of curves and the formula \eqref{eq:kappa},  
we infer that the number 
\[
	\kappa_v = Av\,\cdotp v = f_{xx}(0)v_1^2 + 2f_{xy}(0,0)v_1v_2 + f_{yy}(0)v_2^2  
\]
is the signed curvature of the curve $C_v$ at the point $(0,0)$. 

On the unit circle $|v|^2=v_1^2+v_2^2=1$ the quadratic form $v\mapsto Av\,\cdotp v$
reaches its maximum $\kappa_1$ and minimum $\kappa_2$; these are the {\em principal curvatures} of 
the surface \eqref{eq:graphS} at $(0,0)$. Since $A$ is symmetric, 
$\kappa_1$ and $\kappa_2$ are its eigenvalues. The real numbers
\begin{equation}\label{eq:curvatures}
	H=\kappa_1+\kappa_2 = \mathrm{trace}\, A, \qquad K=\kappa_1 \kappa_2 =\det A
\end{equation}
are, respectively, the {\em mean curvature} and the {\em Gaussian curvature} of $S$ at $(0,0,0)$. 

Note that the trace of $A$ \eqref{eq:Hessian} equals the Laplacian $\Delta f (0,0)$. On the other hand, 
the trace of a matrix is the sum of its eigenvalues. This implies
\begin{equation}\label{eq:LaplaceisH}
	\Delta f (0,0) =\kappa_1+\kappa_2 = H.
\end{equation}

%
%
\begin{lem}\label{lem:Laplace-MC}
Let $D$ be a domain in $\R^2$. If $X: D\to \R^n$ is a smooth conformal immersion, then
for every $p\in D$ the vector $\Delta X(p)$ is orthogonal to the plane $dX_p(\R^2)\subset \R^n$. Equivalently, the following identities hold on $D$:
\begin{equation}\label{eq:Laplaceorthogonal}
	\Delta X \,\cdotp X_u=0,\qquad 	\Delta X \,\cdotp X_v=0. 
\end{equation}
\end{lem}

\begin{proof}
Recall from \eqref{eq:conformal} that $X$ is conformal if and only if $X_u\,\cdotp X_u=X_v\,\cdotp X_v$ 
and $X_u\,\cdotp X_v=0$. 
Differentiating the first identity on $u$ and the second one on $v$ yields
\[
	 X_{uu} \,\cdotp  X_{u} =  X_{uv} \,\cdotp  X_{v} = - X_{vv} \,\cdotp  X_{u},
\]
whence $\Delta X \, \cdotp  X_{u} = (X_{uu} + X_{vv})\, \cdotp  X_{u} = 0$. 
Likewise, differentiating the first identity on $v$ and
the second one on $u$ gives $\Delta X \, \cdotp  X_{v} =0$. 
\end{proof}

We can now prove the following result due to Meusnier \cite{Meusnier1776}.

%
%
\begin{thm}\label{th:CMI}
A smooth conformal immersion $X=(x,y,z): D\to \R^3$ from a domain 
$D\subset\R^2$ parameterizes a surface with vanishing mean curvature function 
if and only if the map $X$ is harmonic, $\Delta X=(\Delta x,\Delta y,\Delta z)=0$.
\end{thm}

\begin{proof}
Fix a point $p_0\in D$; by a translation of coordinates we may assume that  $p_0=(0,0)\in \R^2$.
Since the differential $dX_{(0,0)}:\R^2\to\R^3$ is a conformal linear map, we may assume 
up to a rigid motion on $\R^3$ that $X(0,0)=(0,0,0)$ and 
\[
	dX_{(0,0)}(\xi_1,\xi_2)=\mu(\xi_1,\xi_2,0)\ \ 
	\text{for all $\xi=(\xi_1,\xi_2)\in\R^2$} 
\]	
for some $\mu>0$. Equivalently, at $(u,v)=(0,0)$ the following hold: 
\begin{equation}\label{eq:at0}
	 x_u=y_v=\mu>0,\quad x_v=y_u=0,\quad  z_u=z_v=0.
\end{equation}
Note that 
\begin{equation}\label{eq:mu}
	\mu = |X_u| = |X_v| = \frac{1}{\sqrt2} |\nabla X|.
\end{equation}
The implicit function theorem shows that there is a neighbourhood $U\subset D$ of the origin
such that the surface $S=X(U)$ is a graph $z=f(x,y)$ with $df_{(0,0)}=0$, so 
$f$ is of the form \eqref{eq:graphS}. 
Since the immersion $X$ is conformal, \eqref{eq:Laplaceorthogonal} shows that 
$\Delta X$ is orthogonal to the $(x,y)$-plane $\R^2\times \{0\}$ at the origin, which means that 
\begin{equation}\label{eq:Laplaceat0}
	\Delta x = \Delta y = 0\ \ \text{at $(0,0)$}. 
\end{equation}
We now calculate $\Delta z(0,0)$. Differentiation of $z(u,v)= f(x(u,v),y(u,v))$ gives
\[
	z_u=f_x x_u+f_y y_u,\qquad z_v=f_x x_v+f_y y_v,
\]
\[
	z_{uu} = \left(f_x x_u+f_y y_u\right)_u 
	= f_{xx} x_u^2 + f_{xy}x_uy_u + f_x x_{uu} + f_{yx}x_u y_u + f_{yy}y_u^2 + f_y y_{uu}. 
\]
At the point $(0,0)$, taking into account \eqref{eq:at0} and $f_x=f_y=0$
we get $z_{uu}=\mu^2  f_{xx}$. A similar calculation gives $z_{vv}=\mu^2  f_{yy}$ at $(0,0)$, 
so we conclude that
\begin{equation}\label{eq:Laplacezf}
	\Delta z(0,0) = \mu^2 \Delta f(0,0) = \mu^2 H, 
\end{equation}
where $H$ is the mean curvature of $S$ at the origin (see \eqref{eq:LaplaceisH}). 
Denoting by $\bN=(0,0,1)$ the unit normal vector to $S$ at $0\in\R^3$,
it follows from \eqref{eq:mu}, \eqref{eq:Laplaceat0} and \eqref{eq:Laplacezf} that 
\begin{equation}\label{eq:LaplaceH}
	\Delta X 
	=  \frac12 |\nabla X|^2 H \bN
\end{equation}
holds at $(0,0)\in D$. In particular, $\Delta X=0$ if and only if $H=0$. 
This formula is clearly independent of the choice of a Euclidean coordinate system.
\end{proof}

Combining Theorems \ref{th:harmonic} and \ref{th:CMI} gives:

\begin{cor}
Let $D$ be a relatively compact domain in $\R^2$ with piecewise smooth boundary.
A smooth conformal immersion $X:\overline D\to\R^3$ 
is a stationary point of the area functional if and only if the immersed
surface $S=X(D)$ has vanishing mean curvature at every point.
\end{cor}

Although we used conformal parameterizations, 
neither curvature nor area depend on the choice of parameterization.
This motivates the following definition. 

%
%
\begin{defin}\label{def:minimalsurface}
A smooth surface in $\R^3$ is a {\em minimal surface} if and only if its mean curvature vanishes 
at every point. 
\end{defin}

Every point in a minimal surface is a saddle point, and the surface 
is equally curved in both principal directions but in the opposite normal directions. 
Furthermore, the Gaussian curvature $K= \kappa_1 \kappa_2=-\kappa_1^2\le 0$ 
is nonpositive at every point. The integral
\begin{equation}\label{eq:TC}
	\TC(S)=\int_S K\,\cdotp dA \in [-\infty,0]
\end{equation}
of the Gaussian curvature function with respect to the surface area on $S$ is called 
the {\em total Gaussian curvature}. This number equals zero if and only if $S$ is a 
piece of a plane.

The results presented in this section easily extend to surfaces in $\R^n$ for any $n\ge 3$
which are parameterized by conformal immersions $X:M\to\R^n$ from any open Riemann surface $M$. 
(By the maximum principle for harmonic maps, there are no compact minimal surfaces in $\R^n$.)
There is a sphere $S^{n-3}$ of unit normal vectors to the surface at a given point,
and one must consider the mean curvature of the surface in any given normal direction.
This gives the mean curvature vector field $\bH$ along the surface, which is orthogonal to it
at every point. For surfaces in $\R^3$ we have $\bH=H\bN$, where $H$ is the
mean curvature function \eqref{eq:curvatures} and $\bN$ is a unit normal vector field to the surface.
The formula \eqref{eq:LaplaceH} can then be written in the form  
\[
	\frac{2}{|\nabla X|^2}  \Delta X = \Delta_g X = \bH,
\]
where $\Delta_g X$ denotes the intrinsic Laplacian of the map $X$ with respect to
the induced metric $g=X^*ds^2$ on the surface $M$ 
(cf.\ \cite[Lemma 2.1.2]{AlarconForstnericLopez2021}). 
The formula \eqref{eq:firstvariationD} for the first variation
of area still holds. It shows that the mean curvature vector field $\bH$ is the negative
gradient of the area functional, and the surface is a minimal surface if and only if $\bH=0$.
We refer to \cite{Lawson1980,Osserman1986,AlarconForstnericLopez2021} or any
other standard source for the details.

%
%
\subsection{The Enneper--Weierstrass representation}\label{ss:EW}

In this section we explain the Enneper-Weierstrass formula, which provides a 
connection between holomorphic maps $D\to\C^n$ with special properties
from domains $D\subset \C$ and conformal minimal immersions $D\to\R^n$ for $n\ge 3$.
The same connection holds more generally for maps from any open Riemann surface.

Let $z=x+\imath y$ be a complex coordinate on $\C$. Let us recall the following basic 
operators of complex analysis, also called {\em Wirtinger derivatives}:
\[
	\frac{\di}{\di z}= \frac12 \left(\frac{\di}{\di x} - \imath \frac{\di}{\di y}\right), \qquad
	\frac{\di}{\di \bar z}= \frac12 \left(\frac{\di}{\di x} + \imath \frac{\di}{\di y}\right).
\]
The differential of a function $F(z)$ can be written in the form
\[
	dF = \frac{\di F}{\di x}dx +  \frac{\di F}{\di y}dy 
	= \frac{\di F}{\di z}dz + \frac{\di F}{\di \bar z}d\bar z, 
\]
where $dz=dx+\imath dy$ and $d\bar z=dx-\imath dy$. 
Note that $\frac{\di F}{\di z}dz$ is the $\C$-linear part and $\frac{\di F}{\di \bar z}d\bar z$ is the 
$\C$-antilinear part of $dF$. In particular, $\di F/\di \bar z=0$ 
holds for holomorphic functions, and $\di F/\di z=0$ holds for antiholomorphic ones. 
In terms of these operators, the Laplacian equals
\[
	\Delta =  \frac{\di^2}{\di x^2}+\frac{\di^2}{\di y^2} 
	= 4 \frac{\di}{\di \bar z} \frac{\di}{\di z} =  4 \frac{\di}{\di z}\frac{\di}{\di \bar z}.
\]
Hence, a function $F:D\to\R$ is harmonic if and only if $\di F/\di z$ is holomorphic.

It follows that a smooth map $X=(X_1,X_2,\ldots,X_n):D\to \R^n$ is a harmonic immersion if and only if
the map $f=(f_1,f_2,\dots, f_n):D\to \C^n$ with components $f_j=\di X_j/\di z$ 
is holomorphic and the component functions $f_j$ have no common zero. 
Furthermore, conformality of $X$ is equivalent to the following nullity condition:
\begin{equation}\label{eq:nullity}
	f_1^2 + f_2^2 +\cdots+ f_n^2 = 0.
\end{equation}
Indeed, we have that 
$
	4f_j^2 =  \left( X_{j,x} -\imath X_{j,y} \right)^2 = (X_{j,x})^2 - (X_{j,y})^2 - 2\imath X_{j,x} X_{j,y},
$
and hence
\[
	4 \sum_{j=1}^n f_j^2  = |X_x|^2-|X_y|^2 - 2\imath X_x\,\cdotp X_y.
\]
Comparing with the conformality conditions \eqref{eq:conformal} proves the claim. 

Since we know by Theorem \ref{th:harmonic} that a conformal immersion is harmonic 
if and only it parameterizes a minimal surface, this gives the following result.

%
%
\begin{thm}[The Enneper-Weierstrass representation]\label{th:EW}
Let $D$ be a connected domain in $\C$.
For every smooth conformal minimal immersion $X=(X_1,X_2,\ldots,X_n):D\to\R^n$,
the map $f=(f_1,f_2,\ldots,f_n)=\di X/\di z : D\to \C^n \setminus \{0\}$ 
is holomorphic and satisfies the nullity conditions \eqref{eq:nullity}. 
Conversely, a holomorphic map $f:D\to \C^n\setminus \{0\}$ 
satisfying \eqref{eq:nullity} and the period vanishing conditions
\begin{equation}\label{eq:Rperiods}
	\Re \oint_C f\, dz=0\ \ \text{for every closed curve $C\subset D$}
\end{equation}
determines a conformal minimal immersion $X:D\to\R^n$ given by
\begin{equation}\label{eq:primitive}
	X(z)= c + 2\Re \int_{z_0}^z f(\zeta)\, d\zeta,\quad z\in D
\end{equation}
for any base point $z_0\in D$ and vector $c\in\R^n$. 
\end{thm}

Conditions \eqref{eq:Rperiods} guarantee that the integral
in \eqref{eq:primitive} is well-defined, that is, independent of the path of integration. 
The imaginary components 
\begin{equation}\label{eq:flux}
	\Im \oint_C f\, dz  = \pgot(C) \in \R^n
\end{equation}
of the periods define the {\em flux homomorphism} $\pgot:H_1(D,\Z)\to\R^n$
on the first homology group of $D$. Indeed, by Green's formula 
the period $\oint_C f\, dz$ only depends on the homology
class $[C]\in H_1(D,\Z)$ of a closed path $C\subset D$. 

%
%
\begin{rem}[The first homology group] 
If $D$ is a domain in $\R^2\cong \C$ then its first homology group $H_1(D,\Z)$ 
is a free abelian group $\Z^\ell$ $(\ell\in \{0,1,2,\ldots,\infty\})$ with finitely 
or countably many generators. If $D$ is bounded, connected, and its boundary $bD$ 
consists of $l_1$ Jordan curves $\Gamma_1,\ldots, \Gamma_{l_1}$ and $l_2$ isolated points 
(punctures) $p_1,\ldots, p_{l_2}$, then the group $H_1(D,\Z)$ has
$\ell = l_1+l_2-1$ generators which are represented by loops in $D$ based at any given 
point $p_0\in D$, each surrounding one of the holes of $D$. (By a {\em hole},
we mean a compact connected component of the complement $\C\setminus D$.
A hole which is an isolated point of $\C\setminus D$ is called a {\em puncture}.) 
Indeed, if $\Gamma_1$ is the outer boundary curve of $D$, then every other boundary curve 
$\Gamma_2,\ldots, \Gamma_{l_1}$ of $D$ is contained in the bounded component of
$\C\setminus \Gamma_1$, so it bounds a hole of $D$. Likewise, each of the
points $p_1,\ldots,p_{l_2}$ is a hole (a puncture).  Every hole contributes one generator to $H_1(D,\Z)$. 
The same loops then generate the fundamental group $\pi_1(D,p_0)$ 
as a free nonabelian group, and group $H_1(D,\Z)$ is the abelianisation
of $\pi_1(D,p_0)$. A similar description of the homology group $H_1(D,\Z)$ holds for every surface, 
except that its genus enters the picture as well; 
see \cite[Sect.\ 1.4]{AlarconForstnericLopez2021}. 
For basics on homology and cohomology, see J.\ P.\ May \cite{May1999}. 
\end{rem}

It is clear from Theorem \ref{th:EW} that the following quadric complex hypersurface in $\C^n$ 
plays a special role in the theory of minimal surfaces in $\R^n$: 
\begin{equation}\label{eq:nullquadric}
	\A = \A^{n-1}
	= \bigl\{(z_1,\ldots,z_n) \in\C^n : z_1^2+z_2^2 + \cdots + z_n^2 =0\bigr\}.
\end{equation}
This is called the {\em null quadric } in $\C^n$, and $\A_*=\A\setminus\{0\}$ is 
the {\em punctured null quadric}. Note that $\A$ is a complex cone with
the only singular point at $0$. Theorem \ref{th:EW} says that we get all conformal
minimal surfaces $D\to \R^n$ as integrals of holomorphic maps $f:D\to \A_* \subset \C^n$
satisfying the period vanishing conditions \eqref{eq:Rperiods}.

%
%
\smallskip\noindent 
{\bf The Enneper--Weierstrass representation in $\R^3$.} 
In dimension $n=3$, the null quadric $\A$ admits a 2-sheeted quadratic 
parameterization $\phi:\C^2\to \A$ given by 
\begin{equation}\label{eq:2sheeted}
	\phi(z,w) = \bigl(z^2-w^2,\imath(z^2+w^2),2zw\bigr).
\end{equation}
This map is branched at $0\in\C^2$, and $\phi:\C^2\setminus\{0\}\to \A_*$
is a 2-sheeted holomorphic covering map. It follows that
every conformal minimal immersion $X=(X_1,X_2,X_3):D\to\R^3$ 
can be written in the following form (see \cite{Osserman1986} or 
\cite[pp.\ 107--108]{AlarconForstnericLopez2021}): 
\begin{equation}\label{eq:EWR3}
	X(z) = X(z_0) + 2 \Re \int_{z_0}^z
	\left( \frac{1}{2} \Big(\frac{1}{\ggot}-\ggot\Big),
	 \frac{\imath}{2} \Big(\frac{1}{\ggot}+\ggot\Big),1\right) \di X_3.
\end{equation}
Here,  $\di X=\frac{\di X}{\di z}dz=(\di X_1,\di X_2,\di X_3)$, and 
\begin{equation}\label{eq:CGauss0}
	\ggot =  \frac{\di X_3}{\di X_1 -\imath \, \di X_2} : D \lra \CP^1=\C\cup\{\infty\}
\end{equation}
is a holomorphic map to the Riemann sphere (a meromorphic function on $D$), 
called the {\em complex Gauss map} of $X$. Identifying $\CP^1$ with 
the unit $2$-sphere $S^2\subset \R^3$ by the stereographic projection 
from the point $(0,0,1)\in S^2$, $\ggot$ corresponds to the classical Gauss map 
$\bN= X_x \times X_y/|X_x \times X_y|:D\to S^2$ of $X$. 

Many important quantities and properties of a minimal surface are determined by its Gauss map.
In particular, we have that
\begin{eqnarray*}
	g  &=& X^* ds^2 = 
	2\left(|\di X_1|^2+|\di X_2|^2+|\di X_3|^2\right) = \frac{(1+|\ggot|^2)^2}{4|\ggot|^2} |\di X_3|^2 
	\\
	K g &=&  - \frac{4|d\ggot|^2}{(1+|\ggot|^2)^2}  = - \ggot^*(\sigma^2_{\CP^1}).	
\end{eqnarray*}
Here, $K$ is the Gauss curvature function \eqref{eq:curvatures} of the metric $X^*ds^2$
and $\sigma^2_{\CP^1}$ is the spherical metric on $\CP^1$.
It follows that the total Gaussian curvature (see \eqref{eq:TC}) 
of a conformal minimal surface $X:D\to\R^3$ equals the negative spherical area of the image 
of the Gauss map $\ggot:D\to\CP^1$ counted with multiplicities,
where the area of the sphere $\CP^1=S^2$ is $4\pi$:
\begin{equation}\label{eq:TCarea}
	\TC(X) = - \Area \, \ggot(D).
\end{equation}
It is a recent result that every holomorphic map $D\to \CP^1$ is the complex Gauss map
of a conformal minimal immersion $X:D\to\R^3$; see Theorem \ref{th:Gauss}. 
Hence, the total Gaussian curvature of a minimal surface can be any number in $[-\infty,0]$.

%
%
\begin{exa}[Catenoid]\label{ex:catenoid}
A conformal parameterization of a standard catenoid 
(see \cite[Fig.\ 2.1, p.\ 117]{AlarconForstnericLopez2021})
is given by  the map $X=(X_1,X_2,X_3):\R^2 \to \R^3$,
\begin{equation}\label{eq:catenoid}
	X(u,v)=\left(\cos u \,\cdotp \cosh v, \sin u \,\cdotp \cosh v, v\right).
\end{equation}
It is $2\pi$-periodic in the $u$ variable, hence infinitely-sheeted. 
Introducing the variable $z=\E^{-v+\imath u} \in \C^*$,
we pass to the quotient $\C/(2\pi\,\Z) \cong\C^*$ and obtain a single-sheeted parameterization
$X:\C^* \to \R^3$ having the Enneper--Weierstrass representation
\begin{equation}\label{eq:WRcatenoid}
	X(z) = (1,0,0) - 2\Re \int_1^z \left(\frac12 \Big(\frac1{\zeta}-\zeta\Big),
	\frac{\imath}{2} \Big(\frac1{\zeta}+\zeta\Big),1 \right)\frac{d\zeta}{\zeta}.
\end{equation}
Its Gauss map is $\ggot(z)=z$ extends to the identity map $\CP^1\to \CP^1$.
Hence, by \eqref{eq:TCarea} the catenoid has total Gaussian curvature equal to $-4\pi$.

The catenoid is one of the most paradigmatic examples  in the theory of minimal surfaces.
A compendium of major results about it can be found in 
\cite[Example 2.8.1]{AlarconForstnericLopez2021}.
\end{exa}

%
%
\begin{exa}[Helicoid]\label{ex:helicoid}
A conformal parameterization $X:\R^2\to\R^3$ of the standard left helicoid, 
shown on \cite[Fig.\ 2.2, p.\ 119]{AlarconForstnericLopez2021}, is
\begin{equation}\label{eq:helicoid}
	 X(u,v) = (\sin u \,\cdotp\sinh v, -\cos u \,\cdotp\sinh v, u).
\end{equation}
Its Weierstrass representation in the complex coordinate $z=u+\imath v \in\C$ is
\[
	X(z)= \Re \int_0^z
	\left(\frac12 \left(\frac{1}{\E^{\imath \zeta}} - \E^{\imath \zeta}\right),
	 \frac{\imath}{2} \left(\frac{1}{\E^{\imath \zeta}} + \E^{\imath \zeta}\right), 1 \right) d\zeta.
\]
Its complex Gauss map $\ggot(z)=\E^{\imath z}$ is transcendental, so 
the helicoid has infinite total Gaussian curvature $-\infty$. 
Changing the sign of the second component in \eqref{eq:helicoid} gives a right helicoid.
Like the catenoid, the helicoid is a paradigmatic example satisfying various uniqueness theorems.
E.\ Catalan \cite{Catalan1842} proved in 1842 that the helicoid and the plane are the
only ruled minimal surfaces in $\R^3$, i.e., unions of straight lines. Much more recently, W.\ H.\ Meeks
and H.\ Rosenberg proved in 2005 \cite{MeeksRosenberg2005AM} that the helicoid and the plane
are the only properly embedded, simply connected minimal surfaces in $\R^3$. 
Their proof uses curvature estimates of T.\ H.\ Colding and W.\ P.\ Minicozzi 
\cite{ColdingMinicozzi2004-IV}. 
\end{exa}

%
%
\begin{rem} [Branch points]
Our definition of a conformal map $X:D\to \R^n$ of class $\Ccal^1(D)$ 
requires that equations \eqref{eq:conformal} hold. We have already observed 
that such a map has rank zero at non-immersion points. Assuming that $X$ is 
harmonic at immersion points, it follows that $f=\di X/\di z :D\to \C^n$ is a continuous map
with values in the null quadric $\A$ \eqref{eq:nullquadric} which is holomorphic at immersion points 
of $X$ and vanishes at non-immersion points. By a theorem of T.\ Rad\'o \cite{Rado1924}
(cf.\ \cite[Theorem 15.1.7]{Rudin2008}), such $f$ is holomorphic everywhere on $D$, 
and in particular its zero set consists of isolated points (assuming that $X$ and hence $f$
are nonconstant). This shows that the minimal surface parameterized by $X$ 
has only isolated singularities. See \cite{Tromba2012} for more details.

There are interesting examples of minimal surfaces with branch points.
For example, {\em Henneberg's surface} (see \cite[Example 2.8.9]{AlarconForstnericLopez2021})
is a complete non-orientable minimal surface with two branch points (a branched minimal M\"obius strip),
named after Ernst Lebrecht Henneberg \cite{Henneberg1876} who first
described it in his doctoral dissertation in 1875. 
It was the only known non-orientable minimal surface until 1981
when W.\ H.\ Meeks \cite{Meeks1981} discovered a properly immersed 
minimal M\"obius strip in $\R^3$. A properly embedded minimal M\"obius strip in $\R^4$
was found in 2017 \cite[Example 6.1]{AlarconForstnericLopezMAMS}.
\end{rem}

%
%
\subsection{Holomorphic null curves}\label{ss:null}
There is a family of holomorphic curves in $\C^n$ which are close relatives
of conformal minimal surfaces in $\R^n$.
%
%
A holomorphic map $Z=(Z_1,\ldots,Z_n):D\to \C^n$ for $n\ge 3$ from a domain $D\subset \C$
satisfying the nullity condition
\[
	(Z'_1)^2 + (Z'_2)^2 + \cdots + (Z'_n)^2 =0
\]
is a {\em holomorphic null curve}\index{holomorphic null curve} in $\C^n$.
Its complex derivative $f=Z'$ assumes values in the null quadric  $\A$ \eqref{eq:nullquadric},
and we have $\oint_C fdz= \oint_C dZ=0$ for any closed curve $C\subset D$.
Conversely, a holomorphic map $f:D\to\A$ satisfying the period vanishing conditions 
\begin{equation}\label{eq:Cperiods}
	\oint_C fdz =0\quad \text{for every closed curve $C\subset D$}
\end{equation}
integrates to a holomorphic null curve
\begin{equation}\label{eq:nullcurve}
	Z(z)=c + \int_{z_0}^z f(\zeta)d\zeta,\qquad z\in D,
\end{equation}
where $z_0\in D$ is any given base point and $c\in\C^n$. Indeed, 
conditions \eqref{eq:Cperiods} guarantee that the integral in \eqref{eq:nullcurve} 
is independent of the choice of a path of integration.
These period conditions are trivial on a simply connected domain $D$.

%
%
If $Z=X+\imath Y:D\to \C^n$ is an immersed holomorphic null curve, then its real part
$X=\Re Z:D\to\R^n$ and imaginary part $Y=\Im Z:D\to\R^n$ are conformal minimal surfaces
which are harmonic conjugates of each other. 
Indeed, denoting the complex variable in $\C$ by $z=x+\imath y$,
the Cauchy-Riemann equations imply
\[
	f=Z' = Z_x = X_x+\imath Y_x= X_x-\imath X_y = 2 \frac{\di X}{\di z}.
\]
Since $f=Z':D\to\A^{n-1}_*$ satisfies the nullity condition \eqref{eq:nullity}, 
$X$ is a conformal minimal immersion.
In the same way we find that $f=Z'=Y_y+\imath Y_x= 2\imath Y_z$,
so $Y$ is a conformal minimal immersion. Being harmonic conjugates, 
$X$ and $Y$ are called {\em conjugate minimal surfaces}.
Conformal minimal surfaces in the $1$-parameter family
\[
	X^{\,t}=\Re(\E^{\imath \, t}Z) : D\to \R^n,\quad \ t\in\R
\]
are called {\em associated minimal surfaces} of the holomorphic null curve $Z$. 

Conversely, if $X:D\to\R^n$ is a conformal minimal surface and the holomorphic map 
$f=2 \frac{\di X}{\di z}:D\to \A^{n-1}$ satisfies period vanishing conditions 
\eqref{eq:Cperiods}, then $f$ integrates to a holomorphic null curve $Z:D\to\C^n$ \eqref{eq:nullcurve}
with $\Re Z=X$. In general, the imaginary parts of the periods \eqref{eq:nullcurve} determine 
the flux homomorphism $H_1(M,\Z)\to \R$ of the minimal surface $X$ (see \eqref{eq:flux}); 
hence, $X$ is the real part of a holomorphic null curve 
if and only if it has vanishing flux. The periods \eqref{eq:Cperiods} always vanish 
on a simply connected domain $D$, and hence every conformal minimal immersion 
$D\to\R^n$ is the real part of a holomorphic null curve $D\to\C^n$.

The relationship between conformal minimal surfaces and holomorphic null curves extends
to maps having (isolated) branch points. 

%
%
\begin{exa}[Helicatenoid] \label{ex:helicatenoid}
Consider the holomorphic immersion $Z:\C\to\C^3$, 
\begin{equation}\label{eq:helicatenoid}
	Z(z) = (\cos z,\sin z,-\imath z)\in \C^3, \quad\  z=x+\imath y\in\C.
\end{equation}
We have that 
\[
	Z'(z)=(-\sin z,\cos z,-\imath),\quad\ \sin^2 z + \cos^2 z + (-\imath)^2=0.
\]
Hence, $Z$ is a holomorphic null curve. Consider the 1-parameter family of 
its associated minimal surfaces in $\R^3$ for $t\in [0,2\pi]$:
\begin{equation}\label{eq:helicatenoidassociated}
    X^{\,t}(z) = \Re\left( \E^{\imath t} Z(z) \right)
     = \cos t \left(
     \begin{matrix} \cos x \,\cdotp \cosh y \cr \sin x \,\cdotp \cosh y \cr y \end{matrix}\right)
      + \sin t \left(
      \begin{matrix} \sin x \,\cdotp\sinh y \cr -\cos x \,\cdotp\sinh y \cr x \end{matrix}\right).
\end{equation}
At $t=0$ and $t=\pi$ we have a catenoid (see Example \ref{ex:catenoid}), 
while at $t=\pm \pi/2$ we have a helicoid  (see Example \ref{ex:helicoid}).
Hence, these are conjugate minimal surfaces in $\R^3$. 
The holomorphic null curve \eqref{eq:helicatenoid} is called {\em helicatenoid}.
\end{exa}

%
%
%
%
\section{A survey of new results}\label{sec:survey}
This section is a survey of recent results in the theory of minimal surfaces in Euclidean 
spaces, which were discussed in my lecture at 8 ECM. A detailed presentation is available in the 
monograph \cite{AlarconForstnericLopez2021} and, for non-orientable surfaces, in the AMS Memoir
\cite{AlarconForstnericLopezMAMS} by Alarc\'on, L\'opez and myself.

%
%
\subsection{Approximation, interpolation, and general position theorems}\label{ss:approximation}
Holomorphic approximation is a central topic in complex analysis. Holomorphic functions 
and maps with interesting properties are often constructed inductively, exhausting 
the manifold by an increasing sequence of compact sets 
such that one can approximate holomorphic functions uniformly on each one
by holomorphic functions on $M$. The quintessential example is Runge's theorem 
from 1885 \cite{Runge1885} on approximation of holomorphic functions
on a compact set $K\subset \C$ with connected complement by holomorphic polynomials.
A major extension is Mergelyan's theorem \cite{Mergelyan1951} from 1951.

In order to generalize Runge's theorem, we need the following concept. 
Denote by $\Ocal(M)$ the algebra of holomorphic functions on a complex manifold $M$. 
Given a  compact set $K$ in $M$, its $\Ocal(M)$-convex hull (or holomorphic hull) is the set
\[
	\wh K=\bigl\{z\in M: |f(z)|\le \sup_K |f| \ \ \text{for all} \  f\in \Ocal(M)\bigr\}.
\]
If $K=\wh K$ then $K$ is said to be {\em holomorphically convex}, or $\Ocal(M)$-convex, or a
{\em Runge compact}. If $M$ is the complex plane or, more generally, an open Riemann surface, 
then the hull $\wh K$ is the union of $K$ and all relatively compact connected components
of $M\setminus K$ (the holes of $K$ in $M$). There is no topological characterization
of the hull in higher dimensional complex manifolds.

Holomorphically convex sets are the natural sets for holomorphic
approximation. Runge's theorem was extended to open Riemann surfaces by 
H.\ Behnke and K.\ Stein \cite{BehnkeStein1949} in 1949, who proved that any holomorphic function
on a neighborhood of a Runge compact $K$ in open Riemann surface
$M$ can be approximated uniformly on $K$ by holomorphic functions on $M$.
A related result on higher dimensional complex manifolds is the Oka--Weil theorem
which pertains to Runge compacts in $\C^n$ and, more generally,
in any {\em Stein manifold} (a closed complex submanifold of a Euclidean space $\C^n$). 
A recent survey of holomorphic approximation theory can be found in
\cite{FornaessForstnericWold2020}.

We have seen in Subsection \ref{ss:EW} that every conformal minimal immersion
$M\to \R^n$ from an open Riemann surface $M$ is the integral of a holomorphic
map $f:M\to\A_*\subset \C^n$ into the punctured null quadric  $\A_*$; furthermore,
$f$ must satisfy the period vanishing conditions \eqref{eq:Rperiods}.
Hence, a Runge-type approximation theorem for conformal minimal surfaces in $\R^n$
(or holomorphic null curves in $\C^n$) reduces to the approximation problem for holomorphic maps $f:M\to\A_*$ satisfying the period vanishing conditions \eqref{eq:Rperiods} (or \eqref{eq:Cperiods}
when considering null curves).
This is a nonlinear approximation problem. The first part, ignoring the period conditions,
fits within Oka theory. In particular, the manifold $\A_*$ is easily seen to be 
a homogeneous space of the complex orthogonal group $O_n(\C)$.
Runge-type approximation theorems for holomorphic maps from Stein manifolds 
to complex homogeneous manifolds were proved by Hans Grauert \cite{Grauert1957II}
(1957) and Grauert and Kerner \cite{GrauertKerner1963} (1963). 
More generally, a complex manifold $Y$ is said to be an {\em Oka manifold} if and only if
approximation results of this type hold for holomorphic maps $M\to Y$ from any 
Stein manifold in the absence of topological obstructions. Oka theory also includes
interpolation theorems for holomorphic maps, generalizing classical theorems 
of K.\ Weierstrass \cite{Weierstrass1885} and H.\ Cartan \cite{Cartan1953}.
For the theory of Oka manifolds, see \cite{Forstneric2017E}. 

The second part of the problem, ensuring the period vanishing conditions 
\eqref{eq:Rperiods} or \eqref{eq:Cperiods} for holomorphic maps to $\A_*$, can be 
treated by using sprays of holomorphic maps together with elements of convexity theory.
More precisely, Gromov's one-dimensional convex integration lemma 
from \cite{Gromov1973} is useful in this regard. 
The main techniques underlying all subsequent developments
were established in \cite{AlarconForstneric2014IM} (2014). 
Their application led to the following result, which is a summary of several individual
theorems. Parts (i), (ii) and (iv) are due to Alarc\'on, L\'opez, and myself
\cite{AlarconForstneric2014IM,AlarconForstnericLopez2016MZ,AlarconForstnericLopezMAMS}
(the special case of (i) for $n=3$ was obtained beforehand in 
\cite{AlarconLopez2012JDG}), while (iii) was proved by Alarc\'on and Castro-Infantes 
\cite{AlarconCastroInfantes2018GT,AlarconCastro-Infantes2019APDE}.
Related results for conformal minimal surfaces of finite total curvature were given by 
Alarc\'on and L\'opez \cite{AlarconLopez2019}.

%
%
\begin{mainthm} \label{th:approximation} 
Let $K$ be a compact set with piecewise smooth boundary and without holes (a Runge compact)
in an open Riemann surface $M$. Then:
\begin{enumerate}[\rm (i)]
\item Every conformal minimal immersion $X:K \to \R^n\ (n\ge 3)$ 
can be approximated uniformly on $K$ by proper conformal minimal immersions
$\wt X:M\to \R^n$.
\item
The approximating map $\wt X$ can be chosen to have only simple double points if $n=4$,
and to be an embedding if $n\ge 5$. 
\item
In addition, one can prescribe the values of $ \wt X$ on any closed discrete subset of $ M$ 
(Weierstrass-type interpolation). 
\item
The analogous results hold for non-orientable minimal surfaces in $\R^n$
and for holomorphic null curves in $\C^n,\ n\ge 3$.
\end{enumerate}
\end{mainthm}

The proof of Theorem \ref{th:approximation} is fairly complex, and we shall only outline the main idea. 
Fix a nowhere vanishing holomorphic $1$-form $ \theta$ on the open Riemann surface
$M$. (Such a 1-form always exists; see \cite{GunningNarasimhan1967}.) 
By Enneper--Weierstrass (Theorem \ref{th:EW}), it suffices to prove the Runge 
approximation theorem for holomorphic maps $f:M\to \A_*$ satisfying 
the period vanishing conditions \eqref{eq:Rperiods}. 

Consider an inductive step.
Assume that $K\subset L$ are connected Runge compacts with piecewise smooth 
boundaries in $ M$, $X:K\to\R^n$ is a conformal minimal surface, and 
$f=2\di X/\theta : K\to\A_*$. We wish to approximate $X$ by 
a conformal minimal immersion $\wt X:L\to\R^n$. We may assume that $ f(K)$ is not contained in a 
complex ray $\C^*\bz$ of the null quadric  $\A_*$, for otherwise the result is trivial. 
There are two main cases to consider, the noncritical case and the critical case.

%
%
\smallskip
{\em The noncritical case:} there is no change of topology from $ K$ to $ L$.
It is well known that there are closed curves $C_1,\ldots,C_\ell$ in $K$ forming a basis 
of $H_1(K,\Z)$ whose union $C=\bigcup_{j=1}^\ell C_j$ is a Runge compact.
Let $\B^n$ denote the unit ball of $ \C^n$. By using flows of holomorphic vector fields 
on $\C^n$ tangent to $\A$, we construct a smooth map
\[
	 F:K\times \B^{n\ell} \to \A_*, \qquad F(\cdotp,0)=f=2\di X/\theta,
\]
which is holomorphic on $\mathring K\times \B^n$, such that the associated period map
\[
	 \B^{n\ell} \ni t \ \longmapsto\ 
	\left(\int_{C_j} F(\cdotp,t)\theta\right)_{j=1}^\ell \in \C^{n\ell}
\]
is biholomorphic onto its image. Such {\em period dominating spray} can be found of the form
\begin{equation}\label{eq:spray}
	F(p,t)=\phi_{g_1(p)t_1}^{1} \circ \phi_{g_2(p)t_2}^{2} \circ \cdots \circ 
	\phi_{g_{n\ell}(p)t_{n\ell}}^{n\ell} (f(p)) \in\A_*,
	\quad p\in K,
\end{equation}
where each $ \phi^j$ is the flow of a holomorphic vector field tangent to $\A$ and $g_j\in\Ocal(M)$.
We first construct smooth functions $g_i$ on $C$ which give a period dominating spray;
this can be done since the convex hull of $\A$ equals $\C^n$. As $C$ is Runge in $M$,
we can approximate the $g_i$'s by holomorphic functions on $M$, thereby 
obtaining a holomorphic period dominating spray $F$ as above.

In the next key step, we use that $\A_*$ is an Oka manifold, so 
we can approximate $F$ by a holomorphic map 
$\wt F: M\times \B^{n\ell}\to\A_*$. (There is no topological obstruction
since $\A_*$ is connected.) If the approximation is close enough, 
the implicit function theorem furnishes a parameter value $\tilde t\in \B^{n\ell}$ close to $0$ 
such that the map $\tilde f=F(\cdotp, \tilde t):M\to \A_*$ has vanishing real periods
on the curves $C_1,\ldots,C_\ell$. Hence, fixing a point $p_0\in K$, 
the map $\wt X:L\to\R^n$ given by
\[
	\wt X(p)=X(p_0) + \Re \int_{p_0}^p \tilde f\theta,\quad\ p\in L
\]
is a conformal minimal immersion which approximates $X:K\to\R^n$ on $K$.

%
%
\smallskip
{\em The critical case.}
Assume now that $E$ is an embedded smooth arc in $L\setminus \mathring K$ attached with its endpoints 
to $K$ such that $K\cup E$ is a deformation retract of $ L$. (Thus, $L$ has the same topology
as $K\cup E$. This situation arises when passing a critical point of index $1$
of a strongly subharmonic Morse  exhaustion function on $M$.) 
Let $a,b\in bK$ denote the endpoints of $ E$. 
We extend $f$ smoothly across $E$ to a map $f:K\cup E\to \A_*$ such that 
\[
	 \Re \int_E f\theta = X(b)-X(a) \in\R^n.
\]
This is possible since the convex hull of $\A_*$ equals $ \C^n$. 
We then proceed as in the noncritical case: embed $ f$ into a period dominating spray 
of smooth maps $K\cup E\to\A_*$ which are holomorphic on $\mathring K=K\setminus bK$, 
approximate it by a holomorphic spray on $L$ by Mergelyan's theorem, 
and pick a parameter value for which the map in the spray 
has vanishing real periods on $K\cup E$, and hence on $L$. The Enneper--Weierstrass 
formula gives a conformal minimal surface $\widetilde X: L \to \R^n$ approximating 
$X$ on $ K$. 

The proof of the basic approximation theorem (i) (without the properness condition) 
is then completed by induction on a suitable exhaustion of $M$ by Runge compacts, 
alternatively using the above two cases. Critical points of index $2$ do not arise.

Interpolation (part (iii)) is easily built into the same inductive construction.
Indeed, in each of the two cases considered above, we can arrange that none 
of the points $p_j\in M$ at which we wish to interpolate lies on the boundary of $K$ or $L$.
By choosing the functions $g_i$ in the spray $F$ \eqref{eq:spray} to vanish 
at those points $p_j$ which lie in the interior of $K$, we ensure that the spray $F$
is fixed at these points (independent of the parameter $t$), and hence the  
approximating map $\wt X$ will agree with $X$ at these points. 
For each of the finitely many points $p_j\in \mathring L\setminus K$ we choose a 
smooth embedded arc $E_j\subset L\setminus \mathring K$ with 
one endpoint $p_j$ and the other endpoint $q_j\in bK$ such that  
$E_j\setminus \{q_j\}\subset L\setminus K$ and these arcs are pairwise disjoint. 
The set $S=K\cup\bigcup_j E_j$ is then a Runge compact. We extend the map $f:K\to\A_*$
smoothly to $S$ such that for each $j$, $\int_{E_j}f\theta$ has the correct value
which ensures that the integral assumes the prescribed value at $p_j$.
It remains to apply the same method as above with a spray which is period dominating 
also on each of the arcs $E_j$ and to use Mergelyan approximation on the set $S$. 

Properness of the approximating conformal minimal immersion $\wt X:M\to\R^n$ 
(part (ii) of the theorem) requires considerable additional work. The main point is
to prove a relative version of the approximation theorem in part (i) in which all
but two components of the given map $X$ extend to harmonic functions
on all of $M$. One can keep these components fixed while approximating the remaining 
two components such that the resulting map $\wt X$ is a conformal minimal immersion.
This requires a more precise version of the Oka principle. This result is then used
in an inductive scheme which is designed so that $|\wt X(z)|$ tends to infinity as
the point $z\in M$ goes to the ideal boundary of $M$ (i.e., it exists any compact subset).

Finally, the general position theorem in part (ii) uses the same technique together
with the transversality theorem. The details of proof are considerably more involved
from the technical viewpoint, and we shall not deal with this subject here.

%
%
%

%
%
\subsection{Topological structure of spaces of minimal surfaces}\label{ss:structure}

Assume that $M$ is an open Riemann surface. Fix a nowhere vanishing holomorphic 
$1$-form $\theta$ on $M$. Let $n\ge 3$. An immersion $M\to \R^n$ is said to be {\em nonflat }
if its image is not contained in an affine 2-plane.
We introduce the following notation:
\begin{itemize}
\item $\Ocal(M,\A_*)$ and $\Ccal(M,\A_*)$ denote spaces of holomorphic and continuous maps
$M\to \A_*$, respectively.
\item $\CMI(M,\R^n)$ denotes the space of conformal minimal immersions $M\to\R^n$. 
\item $\CMInf(M,\R^n)$ is the subspace of $\CMI(M,\R^n)$ consisting of nonflat  immersions. 
\item $\NC(M,\C^n)$ is the space of holomorphic null immersions $M\to\C^n$. 
\item $\NCnf(M,\C^n)$ is the subspace of $\NC(M,\C^n)$ consisting of nonflat  immersions.
\end{itemize}
Consider the commutative diagram
\[
	\xymatrix{ 
	\NCnf(M,\C^n)  \ar[r]^\phi \ar[d]_{\Re}  &  \Ocal(M,\A_*) \ \ar@{^{(}->}[r]^{\tau}   &  \Ccal(M,\A_*) \\ 
	\Re\NCnf(M,\C^n)   \ar@{^{(}->}[r]^{\iota}   &  \CMInf(M,\R^n)  \ar[u]_\psi}
\]
where
\begin{itemize}
\item 
the maps $\phi:\NCnf(M,\C^n)\to \Ocal(M,\A_*)$ and $\psi:\CMInf(M,\C^n)\to \Ocal(M,\A_*)$
are given by $Z\mapsto \di Z/\theta$ and $X\mapsto 2\di X/\theta$, respectively;
\item
the map $\NCnf(M,\C^n) \to \Re\NCnf(M,\C^n)$ is the projection $Z=X+\imath Y\mapsto X$;
\item
the maps $\iota:\Re\NCnf(M,\C^n)\hra \CMInf(M,\R^n)$ and $\tau:\Ocal(M,\A_*) \hra  \Ccal(M,\A_*)$
are the natural inclusions.
\end{itemize}

Recall that a continuous map $\phi:X\to Y$ between topological spaces is said to be a 
{\em weak homotopy equivalence} if it induces a bijection of path components of the two spaces
and, for each integer $k\in\N$, an isomorphism $\pi_k(\phi):\pi_k(X) \stackrel{\cong}{\to} \pi_k(Y)$
of their $k$-th homotopy groups. The map $\phi$ is a {\em homotopy equivalence} 
if there is a continuous map $\psi:Y\to X$ such that $\psi\circ\phi:X\to X$ 
is homotopic to the identity on $X$ and $\phi\circ\psi:Y\to Y$ is homotopic to the identity on $Y$. 
These notions indicate that the spaces $X$ and $Y$ have the same rough topological shape.

Since $\A_*$ is an Oka manifold, the inclusion $\tau:\Ocal(M,\A_*) \hra  \Ccal(M,\A_*)$ is a weak homotopy
equivalence by the Oka--Grauert principle (see \cite[Corollary 5.5.6]{Forstneric2017E}), 
and by L\'arusson \cite{Larusson2015PAMS} it is a 
homotopy equivalence if $M$ is of finite topological type, i.e., if the homology group 
$H_1(M,\Z)$ is a finitely generated abelian group.

The real-part projection map $\Re:\NCnf(M,\C^n) \to \Re\NCnf(M,\C^n)$
is evidently a homotopy equivalence.

It turns out that all other maps in the above diagram are also weak homotopy equivalences.
The first part of the following theorem was proved by L\'arusson and myself 
in \cite{ForstnericLarusson2019CAG}, and 
the second part was proved by Alarc\'on, L\'opez and myself in \cite{AlarconForstnericLopez2019JGEA}.
Validity of statement (a) for $\CMI(M,\R^n)$ and $\NC(M,\C^n)$ remains an open problem.

%
%
%
\begin{mainthm} \label{th:structure}
Let $M$ be an open Riemann surface.
\begin{enumerate} [\rm (a)]
\item
Each of the maps $\iota$, $\phi$, $\psi$  in the above diagram is a weak homotopy equivalence,
and a homotopy equivalence if $M$ is of finite topological type.
\item 
The map  $\tau\circ \psi:\CMI(M,\R^n)\to \Ccal(M,\A_*)$ induces a bijection of path components
of the two spaces. Hence, 
\[
	 \pi_0(\CMI(M,\R^n)) = \begin{cases} \Z_2^\ell, & n=3,\ H_1(M,\Z)=\Z^\ell; \\
									   0, & n>3.
						    \end{cases}
\]
\end{enumerate}
\end{mainthm}

It follows that each of the spaces $\NCnf(M,\C^n)$ and $\CMInf(M,\C^n)$ 
is weakly homotopy equivalent to the space $\Ccal(M,\A_*)$ of continuous maps $M\to \A_*$, and 
is homotopy equivalent to $\Ccal(M,\A_*)$ if the surface $M$ has finite topological type.

The group $\Z_2=\{0,1\}$, which appears in part (b), is the 
fundamental group of the punctured null quadric $\A_*\subset\C^3$; see \eqref{eq:2sheeted} 
and note that $\C^2\setminus \{0\}$ is simply connected. 
If $X\in \CMI(M,\R^3)$ then $\di X/\di z:M\to \A_*$ maps every generator of 
the homology group $H_1(M,\Z)$ either to the generator of $\pi_1(\A_*)$ 
or to the trivial element. This gives $2^\ell$ choices, each one determining 
a connected component of $\CMI(M,\R^3)$.
The null quadric $\A_*\subset\C^n$ for $n>3$ is simply connected.

These results are proved by using the parametric versions of techniques discussed in 
Subsection \ref{ss:approximation}. Each of the maps in question satisfies the
parametric h-principle, which implies that it is a weak homotopy equivalence.

%
%
%
%
\subsection{The Gauss map of a conformal minimal surface}\label{ss:Gauss}
The Gauss map is of major importance in the theory of minimal surfaces. 
We have already seen that the Gauss map of a conformal minimal immersion
$X:M\to\R^3$ is a holomorphic map $\ggot:M\to\CP^1$ \eqref{eq:CGauss0}, 
which coincides with the classical Gauss map $M\to S^2$ under the stereographic projection from 
$S^2$ onto $\CP^1$. In general for any dimension $n\ge 3$  one defines the 
{\em generalized Gauss map} of a conformal minimal immersion $X=(X_1,X_2,\ldots,X_n): M\to\R^n$ 
as the Kodaira-type holomorphic map
\begin{equation}\label{eq:Gaussmaprn}
	\Gcal=[\di X_1:\di X_2:\cdots:\di X_n] : M\to Q^{n-2} \subset \CP^{n-1},
\end{equation}
where 
\[
	Q=Q^{n-2}=\left\{ [z_1:\cdots:z_n]\in \CP^{n-1} :  \sum_{j=1}^n z_j^2=0\right\}
\]
is the projectivization of the punctured null quadric $\A_*$,  a smooth quadric complex
hypersurface in $\CP^{n-1}$.
A recent discovery is the following converse result from \cite{AlarconForstnericLopez2019JGEA}
(see also \cite[Theorem 5.4.1]{AlarconForstnericLopez2021}), which shows that every natural
candidate is the Gauss map of a conformal minimal surfaces.

%
%
\begin{mainthm}\label{th:Gauss}
Assume that $n\ge 3$. 
\begin{enumerate}[\rm (i)]
\item For every holomorphic map $\Gcal:M\to Q^{n-2}$ from an open
Riemann surface there exists a conformal minimal immersion $X:M\to\R^n$ with
the Gauss map $\Gcal$.
\item
If $M$ is a compact bordered Riemann surface and $\Gcal:M\to Q^{n-2}$ is a map of class
$\Acal^{r-1}(M, Q^{n-2})$ for some $r\in\N$, then there is a conformal minimal immersion $X:M\to\R^n$
of class $\Ccal^r(M,\R^n)$ with the Gauss map $\Gcal$.
\end{enumerate}
\end{mainthm}

Here, $\Acal^{r-1}(M, Q^{n-2})$ denotes the space of maps $M\to Q^{n-2}$ of class $\Ccal^{r-1}$
which are holomorphic in the interior $M\setminus bM$ of $M$.

Furthermore, the following assertions hold true in both cases in the above theorem.
\begin{enumerate}[\rm (i)]
\item The conformal minimal immersion $X$ can be chosen to have vanishing flux.
In particular, every holomorphic map $\Gcal:M\to Q^{n-2}$ is the Gauss map of
a holomorphic null curve $M\to\C^n$.
\item If $\Gcal(M)$ is not contained in any projective hyperplane of $\CP^{n-1}$,
then $X$ can be chosen with arbitrary flux, to have prescribed values on a given closed 
discrete subset $\Lambda$ of $M$, to be an immersion with simple double points if $n = 4$, 
and to be an injective immersion if $n\ge 5$ and the prescription of values on $\Lambda$ is injective.
\end{enumerate}

When $n=3$, the quadric $Q^1$ is an embedded rational curve in $\CP^2$ parameterized 
by the biholomorphic map
\begin{equation}\label{eq:parameterizationQ1}
	\CP^1 \ni t \ \stackrel{\tau}{\longmapsto} \
	\left[ \frac{1}{2} \Big(\frac{1}{t}-t \Big) :
	 \frac{\imath}{2} \Big(\frac{1}{t}+t\Big):1\right]
	 = \left[1-t^2: \imath(1+t^2):2t\right]
	 \in  Q^1.
\end{equation}
Writing $(1-t^2, \imath(1+t^2),2t)=(a,b,c)$, we easily find that
\[ 
	t= \frac{c}{a-\imath\, b} = \frac{b-\imath \, a}{\imath\, c} \in \CP^1.
\]
Suppose that $X=(X_1,X_2,X_3):M\to\R^3$ is a conformal minimal immersion, and write
$
	2 \partial X = 2(\di X_1,\di X_2,\di X_3) =(\phi_1,\phi_2,\phi_3).
$
In view of the above formula for $t=t(a,b,c)$ it is natural to consider the holomorphic map
\[ 
	\ggot = \frac{\phi_3}{\phi_1-\imath \, \phi_2}
	=  \frac{\di X_3}{\di X_1-\imath \, \di X_2} : M \lra \CP^1.
\] 
This is the complex Gauss map \eqref{eq:CGauss0} of $X$, which appears in the 
Enneper--Weierstrass representation \eqref{eq:EWR3}. 
The generalized Gauss map $\Gcal:M\to  Q^1\subset\CP^2$ \eqref{eq:Gaussmaprn}
of $X$ is then expressed by $\Gcal=\tau\circ \ggot$,  
where $\tau:\CP^1\to Q^1$ is given by \eqref{eq:parameterizationQ1}.

Let us say a few words about the proof of Theorem \ref{th:Gauss}. 
The first step is to lift the given map $\Gcal:M\to Q$ to a holomorphic map
$G:M\to \A_*$. Note that the natural projection $\A_*\to Q$
sending $(z_1,\ldots,z_n)$ to $[z_1:\cdots:z_n]$ 
is a holomorphic fibre bundle with fibre $\C^*=\C\setminus \{0\}$. The existence of a continuous lifting
follows by noting that the homotopy type of $M$ is a wedge of circles, and every
oriented $\C^*$-bundle over a circle is trivial. Further, since $\C^*$ is an Oka manifold,
every continuous lifting is homotopic to a holomorphic lifting
according to the Oka principle \cite[Corollary 5.5.11]{Forstneric2017E}.

In the second and main step of the proof, the holomorphic map 
$G:M\to \A_*$ is multiplied by a nowhere vanishing holomorphic
function $h:M\to \C^*$ such that the product $f=hG:M\to \A_*$ has vanishing
periods along closed curves in $M$ (see \eqref{eq:Cperiods}), and hence it integrates
to a holomorphic null immersion $Z:M\to\C^n$. Its real part $X=\Re Z:M\to\R^n$
is then a conformal minimal immersion having the Gauss map $\Gcal$.
The construction of such a multiplier $h$ follows the idea of proof 
of Theorem \ref{th:approximation}, but the details are fairly nontrivial
and we refer to the cited works.

There are many results in the literature relating the behaviour of a minimal surface
to properties of its Gauss map. A particularly interesting question is how many hyperplanes
in a general position in $\CP^{n-1}$ can be omitted by the Gauss map of a complete 
conformal minimal surface of finite total curvature. A discussion this topic can be found 
in \cite[Chapter 5]{AlarconForstnericLopez2021} and in several other sources.

%
%
%
%
\subsection{The Calabi--Yau problem}\label{ss:CY}
A smooth immersion $X:M\to\R^n$ is said to be complete if $X^*ds^2$ is a complete 
metric on $M$. Equivalent, for every divergent path $\gamma:[0,1)\to M$
(i.e., such that $\gamma(t)$ leaves every compact set in $M$ as $t\to 1$) the image path
$X\circ \gamma:[0,1)\to \R^n$ has infinite Euclidean length.
Clearly, if $X$ is proper then it is complete since any such path $X\circ \gamma(t)$ diverges
to infinity as $t\to 1$. The converse is not true; it is easy to construct complete immersions
(and embeddings if $n\ge 3$) with bounded image $X(M)\subset \R^n$. 

It is however not so easy to find complete bounded immersions with additional properties, such as 
conformal minimal or, in case when the target is a complex Euclidean space $\C^n$,
holomorphic. The following conjecture was posed by Eugenio Calabi in 1965,
\cite[p.\ 170]{Calabi1965Conjecture}. Calabi's conjecture was also promoted by 
S.\ S.\ Chern \cite[p.\ 212]{Chern1966BAMS}.

%
%
\begin{conjecture}\label{conj:Calabi1}
Every complete minimal hypersurface in $\R^n$ $(n\ge 3)$ is unbounded.
Furthermore, every complete nonflat minimal hypersurface in $\R^n$ $(n\ge 3)$ 
has an unbounded projection to every $(n-2)$-dimensional affine subspace.
\end{conjecture}

A particular reason which may have led Calabi to propose these conjectures was
the theorem of S.\ S.\ Chern and R.\ Osserman \cite{ChernOsserman1967JAM} from that time.
Their result says in particular that if $X:M\to\R^n$ $(n\ge 3)$ is a complete conformal 
minimal surface  of finite total Gaussian curvature $\TC(X)>-\infty$, then $M$ is the complement 
of finitely many points $p_1,\ldots,p_m$ in a compact Riemann surface $R$, the
holomorphic $1$-form $\di X$ has an effective pole at each point $p_j$, and $X$ is proper. 
(The first statement holds even without the completeness assumption on $X$,  
due to a result of Huber \cite{Huber1957} from 1957.) The Chern--Osserman theorem 
says that such $X$ is complete if and only if $\di X$ has an effective pole at each puncture $p_j$.
The asymptotic behaviour of $X$ at the punctures was described by 
M.\ Jorge and W.\ Meeks \cite{JorgeMeeks1983T} in 1983.

It turns out that, at least in dimension $n=3$, Calabi's conjecture is both right and wrong, 
depending on whether the minimal surface is embedded
or merely immersed. (This point was not specified in the original question.)
In dimension $n=3$, the answer is radically different for these two cases, 
as we now explain. 

The first counterexample to Calabi's conjecture in the immersed case was given by 
L.\ P.\ de M.\ Jorge and F.\ Xavier in 1980 \cite{JorgeXavier1980AM},
who constructed a complete nonflat conformal minimal immersion $\D\to\R^3$ 
from the disc with the range contained in a slab between two parallel planes.

In 1982, S.-T.\ Yau pointed out in \cite[Problem 91]{Yau1982} that  
the question whether there are complete bounded minimal surfaces in $\R^3$ 
remained open despite Jorge--Xavier's example. This became known as the 
{\em Calabi-Yau problem for minimal surfaces}. 

The problem was resolved for immersed surfaces by N.\ Nadirashvili \cite{Nadirashvili1996IM} 
who in 1996 constructed a complete conformal minimal immersion $\D\to\R^3$ 
with the image contained in a ball. Many subsequent results followed, 
showing similar results for topologically more general surfaces; see 
\cite[Section 7.1]{AlarconForstnericLopez2021} for a survey and references.
However, the conformal type of the examples could not
be controlled by the methods developed in those papers, except for the disc.
The reason is that the increase of the intrinsic radius of a surface was achieved by applying
Runge's theorem on pieces of a suitable labyrinth in the surface, chosen such that any divergent
path avoiding most pieces has infinite length, while crossing a piece of the labyrinth 
increases the length by a prescribed amount. However, Runge's theorem does not allow
to control the map everywhere, and hence small pieces of the surface had to be cut away
in order to keep the image bounded. This surgery changes the conformal
type of the surface, and only its topological type can be controlled by this method.

After Nadirashvili's paper, Yau revisited the Calabi--Yau conjectures in his 
2000 millenium lecture and proposed several new questions  
(see \cite[p.\ 360]{Yau2000AMS} or \cite[p.\ 241]{Yau2000AJM}). He asked in particular: 
What is the geometry of complete bounded minimal surfaces in $\R^3$? 
Can they be embedded? What can be said about the asymptotic behaviour of these surfaces near their ends?

Concerning Calabi's conjecture for embedded surfaces, Colding and Minicozzi showed in 2008 
\cite{ColdingMinicozzi2008AM} that every complete embedded minimal surface in $\R^3$ 
of finite topological type is proper in $\R^3$. Their result was extended to surfaces of finite genus 
and countably many ends by W.\ H.\ Meeks, J.\ P{\'e}rez, and A.\ Ros in 2018,
\cite{MeeksPerezRos-CY}. Hence, 

\smallskip
\noindent {\em Calabi's conjecture holds true for embedded 
minimal surfaces of finite genus and countably many ends in $\R^3$.}
\smallskip

Against this background, we have the following result for immersed surfaces. 

\begin{mainthm}\label{th:CY0}
Every open Riemann surface of finite genus and at most countably many ends,
none of which are point ends, is the conformal structure of a complete 
bounded immersed minimal surface in $\R^3$.
\end{mainthm}

By the uniformization theorem of Z.-X.\ He and O.\ Schramm 
\cite[Theorem 0.2]{HeSchramm1993} (1993) solving Koebe's conjecture,
every open Riemann surface of finite genus and at most countably many ends 
is conformally equivalent to a domain of the form 
\begin{equation}\label{eq:MR}
	M = R\setminus \bigcup_{i} D_i,
\end{equation}
where $R$ is a compact Riemann surface without boundary and 
$\{D_i\}_i$ is a finite or countable family of pairwise disjoint compact geometric discs 
or points in $R$. (A {\em geometric disc} in $R$ is a compact subset 
whose preimage in the universal holomorphic covering space of $R$, which is one of the surfaces
$\CP^1$, $\C$, or $\D$, is a family of pairwise disjoint round discs or points.) 
Such $M$ is called a {\em circled domain} in $R$.
Hence, Theorem \ref{th:CY0} is a corollary to the following
more precise result, which includes information about the boundary behaviour of surfaces.

%
%
\begin{mainthm} 
\label{th:CY}
Assume that $M$ is a circled domain of the form \eqref{eq:MR}. 
For any $n\ge 3$ there exists a continuous map $X:\overline M\to\R^n$ such that
$X:M\to \R^n$ is a complete conformal minimal immersion and $X:bM\to \R^n$ is a
topological embedding. If $n\ge 5$ then there is a topological embedding 
$X:\overline M\to\R^n$ such that $X:M\to \R^n$ is a complete embedded minimal surface.
\end{mainthm}

This means that the image $X(M)$ is a complete immersed minimal surface whose  
boundary $X(bM)$ consists of pairwise disjoint Jordan curves.
The control of conformal structures on complete minimal surfaces 
in Theorems \ref{th:CY0} and \ref{th:CY} is one of the main new aspect of these results; 
the other one is that the surfaces in Theorem \ref{th:CY} have Jordan boundaries. 
These answer the aforementioned questions by Yau.

For surfaces $M$ of type \eqref{eq:MR} with finitely many boundary components, 
Theorem \ref{th:CY} was proved in \cite{AlarconDrinovecForstnericLopez2015PLMS}.
This covers all finite bordered Riemann surfaces in view of the 
uniformization theorem \cite[Theorem 8.1]{Stout1965TAMS} due to E.\ L.\ Stout.
In this case, we actually showed that any conformal minimal immersion
$\overline M\to\R^n$ can be approximated uniformly on $\overline M$ by a 
map $X$ as in the theorem. The general case for countably many ends
was obtained in \cite{AlarconForstneric2021RMI}; an approximation theorem 
also holds in that case.

The situation regarding point ends remains elusive and does not have a clear-cut answer.
On the one hand, a bounded conformal minimal surface cannot be complete at an isolated point end
(a puncture) since a bounded harmonic function extends across a puncture. 
On the other hand, it was shown in \cite[Theorem 5.1]{AlarconForstneric2021RMI}
that an analogue of Theorem \ref{th:CY} holds for connected domains of the form
\[ 
	M = R\setminus \Big(E\cup \bigcup_{i} D_i\Big),
\] 
where $E$ is a compact set in a compact Riemann surface $R$ 
and $D_i\subset R\setminus E$ are pairwise disjoint geometric discs such that the distance 
to $E$ is infinite within $M$. In particular, there are complete bounded conformal minimal surfaces 
in $\R^3$ with point ends which are limits of disc ends. 

Our construction uses an adaptation of the Riemann--Hilbert boundary value problem
to holomorphic null curves and conformal minimal surfaces, together with a 
method of exposing boundary points  of such surfaces. This technique is explained
in detail in \cite[Chapter 6]{AlarconForstnericLopez2021}.
The modifications which we use provide 
a good control of the position of the whole surface in the ambient space, 
thereby keeping it bounded. The main technical  lemma of independent interest
(see \cite[Lemma 7.3.1]{AlarconForstnericLopez2021}) enables one to make the intrinsic radius 
of a conformal bordered minimal surface in $\R^n$ as large as desired by a deformation 
of the surface which is uniformly as small as desired.
One uses this lemma in an inductive process which converges to 
a bounded complete limit surface. This lemma also allows the construction of complete 
minimal surfaces with other interesting geometric properties. In particular, every
bordered Riemann surface admits a complete proper conformal minimal immersion 
into any convex domain in $\R^n$ (embedding if $n\ge 5$) and, more generally, into 
any minimally convex domain (see \cite[Section 8.3]{AlarconForstnericLopez2021}).
A smoothly bounded domain in $\R^3$ is minimally convex if and only if the
boundary has nonnegative mean curvature at each point.

We give a brief description of the modifications which lead to proof of the above results. 
A complete presentation of this technique is given in \cite[Chapter 6]{AlarconForstnericLopez2021}, 
and Theorem \ref{th:CY} is proved in \cite[Chapter 7]{AlarconForstnericLopez2021}. 
Illustrations can be found in my lecture at 
\url{https://8ecm.si/system/admin/abstracts/presentations/000/000/663/original/8ECM2021.pdf?1626190740}. 

Each step consists of two substeps. In the first substep, we choose a large but finite 
number of roughly equidistributed points on the boundary of the surface and change the
surface so that it grows long spikes (tentacles) at these points, 
which however remain uniformly close to the attachment points. 
(Imagine the picture of a corona virus.) 
The effect of this modification is that curves in the surface which terminate near one of the 
exposed boundary points get elongated by a prescribed amount. 
(See \cite[Sect.\ 6.7]{AlarconForstnericLopez2021}.) 

In the second substep, we perform a Riemann--Hilbert type modification 
which increases the intrinsic radius along each of the boundary arcs between a pair of exposed points,
without destroying the effect of substep one. 
To each boundary arc between a pair of exposed points we attach a 3-dimensional
cylinder, consisting of a 1-parameter family of conformal minimal discs centred at points 
of the given arc. The boundaries of these discs form a 2-dimensional cylinder, 
a product of the arc with a circle, and their radii shrink to zero near the exposed endpoints of the 
arc. Is then possible to modify the surface by pushing each arc very near 
the corresponding 2-dimensional cylinder,  with the modification tempering out near the 
exposed endpoints and away from the arcs. So, the modification in substep 2 is big very close to 
the boundary (except near the exposed points), and it is arbitrarily small outside 
a given neighbourhood of the boundary. The new conformal minimal surface is contained in an 
arbitrarily small neighbourhood of the union of the surface from substep 1 and 
the 3-dimensional cylinders that have been attached to the arcs in substep 2.  
The metric effect of the modification in substep 2 is that the length of any path in the surface 
terminating at an interior point of one of the boundary arcs increases almost by the 
radius of the disc that was attached at this point. (For curves terminating near the exposed 
points a desired elongation was already achieved in substep 1.)  
For technical reasons, we actually work with $\di$-derivatives of these conformal minimal
surfaces, including the boundary discs, so the entire picture concerns families 
of holomorphic maps with values in the punctured null quadric $\A_*$. 
In order to control the period conditions, we work with sprays of such configurations, like in the proof of 
Theorem \ref{th:approximation}. Special attention is paid to avoid introducing branch points 
to our surfaces in the process. As said before, this provides the main modification lemma, and
its inductive application leads to the proof of Theorem \ref{th:CY}.

By this method, the Calabi--Yau property has been established in several geometries:
for holomorphic curves in complex manifolds \cite{AlarconForstneric2013MA}, 
holomorphic null curves in $\C^n$ and conformal minimal surfaces in $\R^n$ for $n\ge 3$
\cite{AlarconForstneric2015MA,AlarconDrinovecForstnericLopez2015PLMS,AlarconForstneric2021RMI}, 
holomorphic Legendrian curves in complex contact manifolds
\cite{AlarconForstnericLopez2017CM,AlarconForstneric2019IMRN}, and superminimal
surfaces in self-dual or anti-self-dual Einstein 4-manifolds \cite{Forstneric2021JGA}. 
For a survey and further references,
see \cite[Sect.\ 7.4]{AlarconForstnericLopez2021}. An axiomatic approach to the Calabi--Yau
problem was proposed in \cite{AlarconForstnericLarusson2019X}.

The analogue of the Calabi--Yau problem for complex submanifolds in $\C^n$,
which is known as {\em Paul Yang's problem} who raised it in 1977 \cite{Yang1977JDG}, 
has also received a lot of recent attention. In particular, J. Globevnik showed \cite{Globevnik2015AM} 
that for any pair of integers $1\le k<n$, the ball of $\C^n$ admits holomorphic foliations by complete 
$k$-dimensional proper complex subvarieties, most of which are without singularities
(submanifolds).  Another construction using a different technique was given by Alarc\'on et al.\  
\cite{AlarconGlobevnikLopez2019Crelle}, and it was also shown that there are
nonsingular holomorphic foliations of the ball having complete leaves (Alarc\'on \cite{AlarconJDG}).
Furthermore, there are nonsingular holomorphic foliations  of the ball whose leaves
are complete properly embedded discs \cite{AlarconForstneric2020MZ}. 
The techniques in these papers do not apply to more general minimal surfaces, 
and they do not provide control of complex structures of examples.

In conclusion, I propose the following conjecture. Although I am fully aware of the lack of technical
tools to solve it in this generality, I believe that it is true.

\begin{conjecture}\label{conj:CY}
The Calabi--Yau property holds for bordered minimal surfaces in any smooth  
Riemannian manifold $(N,g)$ with $\dim N\ge 3$. Explicity, for every
bordered Riemann surface, $M$, and conformal minimal immersion
$X:\overline M\to N$ it is possible to approximate $X$ uniformly on $M$ by complete 
conformal minimal immersions $M\to N$.
\end{conjecture}

\begin{acknowledgments}
My research is supported by Program P1-0291 from ARRS, Republic of Slovenia.
I wish to thank Antonio Alarc\'on and Francisco J.\ L\'opez for having introduced me 
to this beautiful subject back in 2011. I also thank the mentioned colleagues, 
as well as Barbara Drinovec Drnov\v sek and Finnur L\'arusson, for the continuing collaboration 
on this subject. 
\end{acknowledgments}


\end{document}